\documentclass[12pt]{amsart}

\usepackage{amsmath,amssymb}
\usepackage{amsfonts}
\usepackage{amsthm}
\usepackage{latexsym}
\usepackage{graphicx}
\usepackage{epsfig}

\newtheorem{theorem}{Theorem}[section]
\newtheorem{lemma}[theorem]{Lemma}

\newtheorem{corollary}[theorem]{Corollary}

\newtheorem{prop}{Proposition}[section]
\newtheorem{rema}[prop]{Remark}

\makeatletter \@addtoreset{equation}{section} \makeatother

\def\ddt{\frac{d}{dt}}
\def\ppt{\frac{\partial}{\partial t}}

\begin{document}

\title{Differential Harnack Estimates for
Time-dependent Heat Equations with Potentials}

\author{Xiaodong Cao $^*$ }
\thanks{$^*$ Research
partially supported by the Jeffrey Sean Lehman Fund from Cornell
University}

\address{Department of Mathematics,
 Cornell University, Ithaca, NY 14853}
\email{cao@math.cornell.edu}

\author{Richard S. Hamilton
}

\address{Department of Mathematics,
 Columbia University, New York, NY 10027}
\email{hamilton@math.columbia.edu}

\renewcommand{\subjclassname}{%
  \textup{2000} Mathematics Subject Classification}

\date{May 8th, 2008}

\maketitle

\markboth{Xiaodong Cao, Richard S. Hamilton } {Differential
Harnack Estimates for Time-dependent Heat Equations with
Potentials}

\begin{abstract}  In this paper, we prove a differential Harnack
inequality for positive solutions of time-dependent heat equations
with potentials. We also prove a gradient estimate for the
positive solution of the time-dependent heat equation.
\end{abstract}

\section{\textbf{Introduction}}
In this paper, we will study time-dependent heat equations with
potentials on closed Riemannian manifolds evolving by the Ricci
flow $$\frac{\partial g_{ij}}{\partial t}=-2R_{ij}.$$ We will
derive differential Harnack inequalities (also known as Li-Yau
type Harnack estimates) for positive solutions of parabolic
equations of the type
$$\frac{\partial f}{\partial t}=\triangle_{g(t)} f+Rf,$$ where
$\triangle_{g(t)}$ depends on time $t$, R is the scalar curvature
of $g(t)$.

The study of differential Harnack estimates for parabolic
equations originated in P. Li and S.-T. Yau's paper \cite{ly86},
in which they proved a differential Harnack inequality for
positive solutions of the heat equation on Riemannian manifolds
with a fixed metric. Namely they proved that, if $f$ is a positive
solution to the heat equation $$\frac{\partial f}{\partial
t}=\triangle f$$ on a Riemannian manifold with nonnegative Ricci
curvature, then
$$\ppt \ln f-|\nabla \ln f|^2+\frac{n}{2t}=\triangle \ln
f+\frac{n}{2t}\geq 0.$$ The idea was later brought to study
general geometric evolution equations by the second author. The
differential Harnack estimates have become an important technique
in the studies of geometric flows.

In \cite{harnack}, the second author proved  a Harnack estimate
for the Ricci flow on Riemannian manifolds with weakly positive
curvature operator, its trace version,
\begin{equation}\label{tharnack}
\frac{\partial R}{\partial t}+\frac{R}{t}+2\nabla R \cdot V+2Rc(V,
V)\geq 0
\end{equation} (here $V$ is any vector field), will be needed in our
proof. B. Chow and S.-C. Chu \cite{chowchu95} gave a nice
geometric interpretation by showing that the Harnack quantity is
the curvature of a degenerate metric in space-time. In the case of
surface, Harnack estimates have been proved by the second author
in \cite{Hsurface} with positive scalar curvature and by B. Chow
in \cite{chow91} with arbitrary curvature, respectively. B. Chow
and the second author generalized their results for the heat
equation and for the Ricci flow on surfaces in
\cite{chowhamilton97}.

 The second author  proved
 a matrix Harnack estimate
for the heat equation in \cite{hmatrix93}. The fundamental
solution and Harnack inequality of time-dependent heat equation
have also been studied by C. Guenther \cite{guenther02}. In
\cite{caoni04}, H.-D. Cao and L. Ni proved a matrix Harnack
estimate for the heat equation on K\"{a}hler manifolds. Harnack
inequalities have also been discovered for other geometric flows.
The second author proved a Harnack estimate for the mean curvature
flow in \cite{mcf95}. B. Chow proved Harnack estimates for
Gaussian curvature flow in \cite{chowgcf91} and for Yamabe flow in
\cite{chowyf92}. For the K\"{a}hler-Ricci flow, H.-D. Cao proved a
Harnack estimate in \cite{caohd92} and L. Ni proved a matrix
Harnack estimate in \cite{ni07}. In \cite{andrews94}, B. Andrews
obtained Harnack inequalities for various evolving hypersurfaces.
In \cite{perelman1}, G. Perelman proved a Harnack estimate for the
fundamental solution of the conjugate heat equation under the
Ricci flow. More precisely, let $(M, g(t))$, $t\in [0, T]$,
 be a solution
to the Ricci flow on a closed manifold, $f$ be the positive
fundamental solution to the conjugate heat equation $$\ppt
f=-\triangle f+Rf,$$ $\tau=T-t$ and $u=-\ln f-\frac{n}{2}\ln (4\pi
\tau)$. Then on $(0,T)$, G. Perelman proved that
$$2\triangle u-|\nabla u|^2+R+\frac{u}{\tau}-\frac{n}{\tau}\leq
0$$ (see \cite{ni06} or \cite[Chapter 16]{chowetc2} for a detailed
proof).

In the present paper, let $(M,g(t))$, $t\in [0, T)$, be a solution
to the Ricci flow on a closed manifold, $f$ be a positive solution
of the time-dependent heat equation with potential, i.e.,
\begin{align}
  \frac{\partial g_{ij}}{\partial t}=& -2R_{ij}, \label{rf} \\
  \frac{\partial f}{\partial t}=&\triangle_{g(t)} f+Rf. \label{heat}
\end{align}
Notice that under the Ricci flow, we have
\begin{equation}
\ddt \int_M f
d\mu=0.
\end{equation}
 We also assume that $g(0)$ has weakly positive curvature
operator, this property is preserved by the Ricci flow (see
\cite{HPCO}). Our first main theorem is the following,
\begin{theorem}\label{theorem1.1} Let $(M, g(t))$, $t\in [0, T)$,
 be a solution
to the Ricci flow (\ref{rf}) on a closed manifold, and suppose
that $g(t)$ has weakly positive curvature operator. Let $f$ be a
positive solution to the heat equation (\ref{heat}), $u=-\ln f$
and $$H=2 \triangle u-|\nabla u|^2-3R-2\frac{n}{t}.$$ Then for all
time $t\in (0,T)$
$$H\leq 0.$$
\end{theorem}
We shall also prove the following theorem,
\begin{theorem}\label{theorem1.2}Let $(M, g(t))$, $t\in [0, T)$,
 be a solution
to the Ricci flow (\ref{rf}) on a closed manifold, and suppose
that $g(t)$ has weakly positive curvature operator. Let $f$ be a
positive solution to the heat equation (\ref{heat}), $v=-\ln
f-\frac{n}{2} \ln (4\pi t)$ and $$P=2 \triangle v-|\nabla
v|^2-3R+\frac{v}{t}-d\frac{n}{t},$$ where $d$ is any constant.
Then for all time $t\in (0,T)$, we have \begin{align*} \ppt
(tP)=&\triangle (tP)-2\nabla (tP) \cdot \nabla
v-2t|v_{ij}-R_{ij}-\frac{1}{2t}g_{ij}|^2 \\& -2t(\triangle
R+2|Rc|^2+\frac{R}{t} +2 \nabla R \cdot \nabla v +2 R_{ij}v_iv_j).
\end{align*} Moreover, $\max tP$ is
non-increasing.
\end{theorem}

\begin{rema}
In \cite{caoche}, the first author uses a similar method and
proves a differential Harnack inequality for all positive
solutions of the conjugate heat equation under the Ricci flow,
notice that Perelman's Harnack inequality is only valid for the
fundamental solution, while the estimate in \cite{caoche} has no
such restriction (but requires nonnegative scalar curvature). Such
Harnack inequality for the conjugate heat equation has also been
proved by S. Kuang and Q. Zhang \cite{kz07} recently.
\end{rema}

As one can see that our equation (1.3) is corresponding to the
heat equation under the Ricci flow, in the sense of (1.4).
Li-Yau's Harnack inequality gives an estimate on the heat kernel.
The Harnack estimate of the Ricci flow on surfaces gives a control
on curvature growth. The general Harnack estimate of the Ricci
flow allows one to classify the ancient solutions of nonnegative
curvature operators. Perelman's Harnack inequality proves
noncollapsing under the Ricci flow. We expect our Harnack estimate
to have similar applications in the study of the Ricci flow. There
are geometric quantities which satisfy equation (1.3), for
example, the scalar curvature on surfaces. Hence our Harnack
estimate should lead more information of the blow-up. We also
expect this estimate will be useful in understanding the Ricci
solitons, as the soliton potential satisfies a heat-type equation.
More importantly, the method used here will help us searching for
more interesting Harnack inequalities in the Ricci flow as well as
in other geometric flows (for example, see \cite{caoche}).

The rest of this paper is organized as follows. In section 2, we
will first derive a general evolution formula for function $H$,
then we give the proof of Theorem \ref{theorem1.1}. We will prove
a integral version of the Harnack inequality (Theorem \ref{intH}).
Finally, we reprove some early results on surfaces. In section 3,
We will derive the general evolution formula for function $P$,
then we prove Theorem \ref{theorem1.2}. In section 4, we will
define two functionals which are associated to the above two
Harnack quantities, and show that they are monotone under the
Ricci flow. Finally, in section 5, as a special case of the
general evolution formula, we will prove a gradient estimate for
positive and bounded solutions to the heat equation under the
Ricci flow.

{\bf Acknowledgement:} The first author thanks Laurent
Saloff-Coste, for many helpful discussion on various aspects of
Harnack inequalities, and for pointing out the reference
\cite{sturm96} to him.

\section{\textbf{Proof of Theorem \ref{theorem1.1}}}

Let us consider positive solutions of
$$\ppt f=\triangle f
-cRf$$ for all constant $c$ which we will fix later. Let
$f=e^{-u}$, then $\ln f=-u$. We have
$$\ppt \ln f=-\ppt u,$$ and
$$\nabla \ln f=-\nabla u, ~\triangle \ln f=-\triangle u.$$
Hence $u$ satisfies the following equation,
\begin{equation}\label{eqnu}
\ppt u=\triangle u-|\nabla u|^2+cR.
\end{equation}

\begin{lemma}\label{lemma:H} Let $(M, g(t))$ be a solution to
the Ricci flow, and
$u$ satisfies (\ref{eqnu}). Let $$H=\alpha \triangle
u-\beta|\nabla u|^2+aR-b\frac{u}{t}-d\frac{n}{t},$$ where
$\alpha$, $\beta$, $a$, $b$ and $d$ are constants that we will
pick later. Then $H$ satisfies the following evolution equation,
\begin{align*}
\ppt H=&\triangle H-2\nabla H \cdot \nabla
u-(2\alpha-2\beta)|u_{ij}-\frac{\alpha}{2\alpha
-2\beta}R_{ij}-\frac{\lambda}{2t}g_{ij}|^2
-\frac{2\alpha-2\beta}{\alpha}\frac{\lambda}{t}H\\&
+(2\alpha-2\beta)\frac{n}{4t^2}\lambda^2
-(b+\frac{2\alpha-2\beta}{\alpha} \lambda \beta)\frac{|\nabla
u|^2}{t}+(1-\frac{2\alpha-2\beta}{\alpha}
\lambda)b\frac{u}{t^2}\\&+(1-\frac{2\alpha-2\beta}{\alpha}
\lambda)d\frac{n}{t^2} +\alpha c \triangle
R+(2a+\frac{\alpha^2}{2\alpha-2\beta})|Rc|^2\\&+(\alpha \lambda
+a\frac{2\alpha-2\beta}{\alpha} \lambda-bc)\frac{R}{t}-2\alpha
R_{ij}u_iu_j+2(a-\beta c) \nabla R \cdot \nabla u,
\end{align*} where $\lambda$ is also a constant that we will pick
later.
\end{lemma}

\begin{proof} The proof follows from a direct computation. We
first calculate the first two terms in $H$,
$$\ppt (\triangle u)=\triangle (\triangle u)-\triangle (|\nabla
u|^2)+c\triangle R+2R_{ij}u_{ij},$$ and
\begin{align*}
\ppt |\nabla u|^2=&2\nabla u \cdot \nabla \triangle u+2Rc(\nabla
u, \nabla u)-2\nabla u\cdot \nabla (|\nabla u|^2)+2c\nabla u
\cdot \nabla R\\
=&\triangle (|\nabla u|^2)-2|\nabla \nabla u|^2-2\nabla u \cdot
\nabla (|\nabla u|^2)+2c\nabla u \cdot \nabla R,
\end{align*}
here we used $$\triangle (|\nabla u|^2)=2\nabla u \cdot \triangle
\nabla u+2|\nabla \nabla u|^2,$$ and $$\triangle \nabla u=\nabla
\triangle u+Rc(\nabla u,\cdot).$$ Using the evolution equation of
$R$,
$$\ppt R=\triangle R + 2|Rc|^2,$$ and (\ref{eqnu}), we have
\begin{align*} \ppt H=&\triangle H -\alpha \triangle (|\nabla
u|^2)+2\alpha R_{ij}u_{ij}+2\beta|\nabla \nabla u|^2+2\beta \nabla
u \cdot \nabla (|\nabla u|^2)\\&+ 2a|Rc|^2+b\frac{|\nabla
u|^2}{t}+d\frac{n}{t^2}+b\frac{u}{t^2}+\alpha c \triangle R-2\beta
c\nabla u
\cdot \nabla R-b\frac{cR}{t}\\
=&\triangle H-2\nabla H \cdot \nabla u+2(a-\beta c) \nabla R \cdot
\nabla u -b\frac{|\nabla u|^2}{t}-2\alpha
R_{ij}u_iu_j\\&-(2\alpha-2\beta)|\nabla \nabla u|^2+ 2\alpha
R_{ij}u_{ij}+2a|Rc|^2+b\frac{u}{t^2}+d\frac{n}{t^2}
+\alpha c \triangle R-b\frac{cR}{t}\\
=&\triangle H-2\nabla H \cdot \nabla u+2(a-\beta c) \nabla R \cdot
\nabla u-2(\alpha-\beta)|u_{ij}-\frac{\alpha R_{ij}}{2(\alpha
-\beta)}-\frac{\lambda g_{ij}}{2t}|^2\\&
-2(\alpha-\beta)\frac{\lambda}{t} (\triangle u-\frac{\alpha
R}{2(\alpha-\beta)})+(\alpha-\beta)\frac{n \lambda^2}{2t^2}
+(2a+\frac{\alpha^2}{2(\alpha-\beta)})|Rc|^2 \\&-b\frac{|\nabla
u|^2}{t}-2\alpha R_{ij}u_iu_j+b\frac{u}{t^2}+d\frac{n}{t^2}
+\alpha
c \triangle R-b\frac{cR}{t}\\
=&\triangle H-2\nabla H \cdot \nabla u+2(a-\beta c) \nabla R \cdot
\nabla u-2(\alpha-\beta)|u_{ij}-\frac{\alpha R_{ij}}{2(\alpha
-\beta)}-\frac{\lambda g_{ij}}{2t}|^2
\\&-\frac{2(\alpha-\beta)}{\alpha}\frac{\lambda}{t}H +(\alpha
\lambda +a\frac{2(\alpha-\beta)}{\alpha}
\lambda-bc)\frac{R}{t}+(\alpha-\beta)\frac{n \lambda^2}{2t^2}\\&
+(2a+\frac{\alpha^2}{2(\alpha-\beta)})|Rc|^2
-(b+\frac{2(\alpha-\beta)}{\alpha} \lambda \beta)\frac{|\nabla
u|^2}{t}-2\alpha R_{ij}u_iu_j\\&+(1-\frac{2(\alpha-\beta)}{\alpha}
\lambda)b\frac{u}{t^2}+(1-\frac{2(\alpha-\beta)}{\alpha}
\lambda)d\frac{n}{t^2} +\alpha c \triangle R.
\end{align*}
\end{proof}

In the above lemma, let us take $\alpha=2$, $\beta=1$, $a=-3$,
$c=-1$, $\lambda=2$, $b=0$, $d=2$. As a consequence of the above
lemma, we have
\begin{corollary} Let $(M, g(t))$ be a solution to the
Ricci flow, $f$ be a positive solution of
$$\ppt f=\triangle f
+Rf,$$ let $u=-\ln f$ and \begin{equation} \label{defH} H=2
\triangle u-|\nabla u|^2-3R-2\frac{n}{t}. \end{equation} Then we
have
\begin{align}\label{harnack}
\ppt H=&\triangle H-2\nabla H \cdot \nabla
u-2|u_{ij}-R_{ij}-\frac{1}{t}g_{ij}|^2 -\frac{2}{t}H
-\frac{2}{t}|\nabla u|^2\\ \nonumber &-2 \triangle R -4|Rc|^2
-2\frac{R}{t}-4 \nabla R \cdot \nabla u-4 R_{ij}u_iu_j\\ \nonumber
=&\triangle H-2\nabla H \cdot \nabla
u-2|u_{ij}-R_{ij}-\frac{1}{t}g_{ij}|^2 -\frac{2}{t}H
-\frac{2}{t}|\nabla u|^2\\ \nonumber &-2(\ppt R +\frac{R}{t}+2
\nabla R \cdot \nabla u+2 R_{ij}u_iu_j).
\end{align}
\end{corollary}

Now we can finish the proof of Theorem \ref{theorem1.1}.

\begin{proof} (Proof of Theorem \ref{theorem1.1}) It is easy to
see that for $t$ small enough that $H(t)<0$. Since $g_{ij}$ has
weakly positive curvature operator, by the trace Harnack
inequality for the Ricci flow proved by the second author in
\cite{harnack},
$$\ppt R +\frac{R}{t}+2 \nabla R \cdot \nabla u+2 R_{ij}u_iu_j\geq
0.$$ It follows from (\ref{harnack}) and maximum principle that
$$H\leq 0$$ for all time $t$.
\end{proof}

\begin{rema}
The theorem is also true on complete non-compact Riemannian
manifolds when we can apply maximum principle. For example, if we
assume that $R$ is uniformly bounded and $\triangle u$ has a upper
bound for all time $t$.
\end{rema}

We now can integrate 
along a space-time path, and we
have the following,

\begin{theorem}\label{intH} Let $(M, g(t))$, $t\in [0, T)$,
 be a solution
to the Ricci flow on a closed manifold, and suppose that $g(t)$
has weakly positive curvature operator. Let $f$ be a positive
solution to the heat equation $$\ppt f=\triangle f +Rf.$$ Assume
that $(x_1, t_1)$ and $(x_2, t_2)$, $0<t_1<t_2$, are two points in
$M\times (0, T)$. Let $$\Gamma=\inf_{\gamma} \int_{t_1}^{t_2}
(|\dot{\gamma}|^2+R) dt,$$ where $\gamma$ is any space-time path
joining $(x_1, t_1)$ and $(x_2, t_2)$. Then we have
$$f(x_1,t_1)\leq f(x_2, t_2) (\frac{t_2}{t_1})^n
\exp^{\Gamma/2}.$$
\end{theorem}

\begin{proof}
Since $H\leq 0$ and $u$ satisfies $$\ppt u=\triangle u-|\nabla
u|^2-R,$$ we have $$2\ppt u+|\nabla u|^2-R-\frac{2n}{t}\leq 0.$$If
we pick a space-time path $\gamma(x,t)$ join $(x_1,t_1)$ and
$(x_2,t_2)$ with $t_2>t_1>0$, then along $\gamma$, we have
\begin{align*}
\ddt u&=\ppt u+\nabla u \cdot \dot{\gamma}\\
&\leq -\frac{1}{2}|\nabla u|^2+\frac{R}{2}+\frac{n}{t}+\nabla u
\cdot \dot{\gamma}\\
&\leq \frac{1}{2}(|\dot{\gamma}|^2+R)+\frac{n}{t}.
\end{align*}
Hence $$u(x_2, t_2)-u(x_1,t_1)\leq \frac12 \inf_{\gamma}
\int_{t_1}^{t_2} (|\dot{\gamma}|^2+R) dt +n\ln
(\frac{t_2}{t_1}).$$ If we denote $\Gamma=\inf_{\gamma}
\int_{t_1}^{t_2} (|\dot{\gamma}|^2+R) dt$, then we have
$$f(x_1,t_1)\leq f(x_2, t_2) (\frac{t_2}{t_1})^n
\exp^{\Gamma/2}.$$ This finishes the proof.
\end{proof}

In the rest of this section, we will restrict ourselves on
surfaces. We will reprove some early results of B. Chow and the
second author in the case of $n=2$. Take $\alpha=1$, $\beta=0$,
$a=c=-1$, $b=d=0$ and $\lambda=0$, let
 $$H=\triangle u -R$$ and
$$H_{ij}=u_{ij}-\frac12 Rg_{ij}.$$ On surfaces, $$\ppt R=\triangle
R+R^2,$$ hence we have
\begin{align}\nonumber
\ppt H=&\triangle H-2\nabla H \cdot \nabla u-2 \nabla R \cdot
\nabla u-2|u_{ij}-\frac{1}{2}R_{ij}|^2 -\frac{3}{2}|Rc|^2 -2
R_{ij}u_iu_j - \triangle R\\ \nonumber =&\triangle
H-2|H_{ij}|^2-2\nabla H \cdot \nabla u-RH-R^2-2 \nabla R \cdot
\nabla u-R|\nabla u|^2-\triangle R\\ \label{eqnsurf} =&\triangle
H-2|H_{ij}|^2-2\nabla H \cdot \nabla u-RH-R|\nabla u+\nabla \ln
R|^2-R(\frac{\partial \ln R}{\partial t}-|\nabla \ln R|^2).
\end{align}
If we further let $f=R$ and $u=-\ln R$, we have
$$\ppt H=\triangle H-2|H_{ij}|^2+2\nabla H \cdot \nabla \ln R.$$
 It follows that $H-\frac{1}{t}\leq 0$. As a consequence, we have

\begin{corollary}(Hamilton \cite{Hsurface}) \label{corohsurf}
If $(M\sp 2, g(t))$ is a solution to the Ricci flow on a closed
surface with $R>0$. The scalar curvature  $R$ satisfies $$\ppt
R=\triangle R +R^2.$$ Then \begin{equation} \label{hh2}\ppt \ln
R-|\nabla \ln R|^2+\frac{1}{t}=\triangle \ln R+R+\frac{1}{t}\geq
0.\end{equation}
\end{corollary}
Using the above Corollary \ref{corohsurf}, plugging into
(\ref{eqnsurf}), we have
\begin{corollary}(Chow-Hamilton \cite{chowhamilton97}) If $(M\sp 2, g(t))$ is
a solution to the Ricci flow on a closed surface with $R>0$, and
$f$ ia a positive solution to $$\ppt f=\triangle f +Rf.$$ Then
\begin{equation} \label{chh2}\ppt \ln f-|\nabla \ln f|^2+\frac{1}{t}=\triangle \ln
f+R+\frac{1}{t}\geq 0.\end{equation}
\end{corollary}

\begin{rema}
When $n=2$, (\ref{hh2}) or (\ref{chh2}) implies our Harnack
estimate in Theorem \ref{theorem1.1}.
\end{rema}

\section{\textbf{Proof of Theorem \ref{theorem1.2}}}

In this section, let $f=(4\pi t)^{-n/2} e^{-v}$, then $\ln
f=-\frac{n}{2} \ln (4\pi t)-v$. We have $$\ppt \ln f=-\ppt
v-\frac{n}{2t},$$ and
$$\nabla \ln f=-\nabla v, ~\triangle \ln f=-\triangle v.$$
Hence $v$ satisfies the following equation,
\begin{equation}\label{eqnv}\ppt v=\triangle v-|\nabla
v|^2+cR-\frac{n}{2t}.\end{equation}

\begin{lemma} Let $(M, g(t))$ be a solution to the Ricci flow, and
$v$ satisfies (\ref{eqnv}). Let
\begin{equation}\label{defP}
P=\alpha \triangle v-|\nabla v|^2+aR-b\frac{v}{t}-d\frac{n}{t},
\end{equation}
where $\alpha$, $a$, $b$ and $d$ are constants that we will pick
later. Then $P$ satisfies
\begin{align*}
\ppt P=&\triangle P-2\nabla P \cdot \nabla v+2(a-c) \nabla R \cdot
\nabla v-(2\alpha-2)|v_{ij}-\frac{\alpha}{2\alpha
-2}R_{ij}-\frac{\lambda}{2t}g_{ij}|^2\\&
-\frac{2\alpha-2}{\alpha}\frac{\lambda}{t}P +(\alpha \lambda
+a\frac{2\alpha-2}{\alpha}
\lambda-bc)\frac{R}{t}+(\alpha-1)\frac{n\lambda^2}{2t^2}
+(2a+\frac{\alpha^2}{2\alpha-2})|Rc|^2\\&
-(b+\frac{2\alpha-2}{\alpha} \lambda)\frac{|\nabla
v|^2}{t}-2\alpha R_{ij}v_iv_j+(1-\frac{2\alpha-2}{\alpha}
\lambda)b\frac{v}{t^2}+b\frac{n}{2t^2}\\&+(1-\frac{2\alpha-2}{\alpha}
\lambda)d\frac{n}{t^2} +\alpha c \triangle R,
\end{align*}where $\lambda$ is also a constant that we will pick
later.
\end{lemma}

\begin{proof} The proof again follows from direct computation.
Recall that we have
$$\ppt (\triangle v)=\triangle (\triangle v)-\triangle (|\nabla
v|^2)+c\triangle R+2R_{ij}v_{ij},$$ and
\begin{align*}
\ppt |\nabla v|^2=&2\nabla v \cdot \nabla \triangle v+2Rc(\nabla
v, \nabla v)-2\nabla v\cdot \nabla (|\nabla v|^2)+2c\nabla v
\cdot \nabla R\\
=&\triangle (|\nabla v|^2)-2|\nabla \nabla v|^2-2\nabla v \cdot
\nabla (|\nabla v|^2)+2c\nabla v \cdot \nabla R,
\end{align*}
here we used $$\triangle (|\nabla v|^2)=2\nabla v \cdot \triangle
\nabla v+2|\nabla \nabla v|^2,$$ and $$\triangle \nabla v=\nabla
\triangle v+Rc(\nabla v,\cdot).$$ Using the evolution equation of
$R$,
$$\ppt R=\triangle R + 2|Rc|^2,$$ and (\ref{eqnv}), we arrive at
\begin{align*} \ppt P=&\triangle P -\alpha \triangle
(|\nabla v|^2)+2\alpha R_{ij}v_{ij}+2|\nabla \nabla v|^2+2\nabla v
\cdot \nabla (|\nabla v|^2)\\&+ 2a|Rc|^2+b\frac{|\nabla
v|^2}{t}+b\frac{n}{2t^2}+d\frac{n}{t^2}+b\frac{v}{t^2}+\alpha c
\triangle R-2c\nabla v
\cdot \nabla R-b\frac{cR}{t}\\
=&\triangle P-2\nabla P \cdot \nabla v+2(a-c) \nabla R \cdot
\nabla v -b\frac{|\nabla v|^2}{t}-2\alpha
R_{ij}v_iv_j-(2\alpha-2)|\nabla \nabla v|^2\\&+ 2\alpha
R_{ij}v_{ij}+2a|Rc|^2+b\frac{v}{t^2}+b\frac{n}{2t^2}+d\frac{n}{t^2}
+\alpha c \triangle R-b\frac{cR}{t}\\
=&\triangle P-2\nabla P \cdot \nabla v+2(a-c) \nabla R \cdot
\nabla v-(2\alpha-2)|v_{ij}-\frac{\alpha}{2\alpha
-2}R_{ij}-\frac{\lambda}{2t}g_{ij}|^2\\&
-(2\alpha-2)\frac{\lambda}{t} (\triangle
v-\frac{\alpha}{2\alpha-2}R)+(2\alpha-2)\frac{n}{4t^2}\lambda^2
+(2a+\frac{\alpha^2}{2\alpha-2})|Rc|^2\\& -b\frac{|\nabla
v|^2}{t}-2\alpha
R_{ij}v_iv_j+b\frac{v}{t^2}+b\frac{n}{2t^2}+d\frac{n}{t^2} +\alpha
c \triangle R-b\frac{cR}{t}\\
=&\triangle P-2\nabla P \cdot \nabla v+2(a-c) \nabla R \cdot
\nabla v-(2\alpha-2)|v_{ij}-\frac{\alpha}{2\alpha
-2}R_{ij}-\frac{\lambda}{2t}g_{ij}|^2\\&
-\frac{2\alpha-2}{\alpha}\frac{\lambda}{t}P +(\alpha \lambda
+a\frac{2\alpha-2}{\alpha}
\lambda-bc)\frac{R}{t}+(\alpha-1)\frac{n\lambda^2}{2t^2}
+(2a+\frac{\alpha^2}{2\alpha-2})|Rc|^2\\&
-(b+\frac{2\alpha-2}{\alpha} \lambda)\frac{|\nabla
v|^2}{t}-2\alpha R_{ij}v_iv_j+(1-\frac{2\alpha-2}{\alpha}
\lambda)b\frac{v}{t^2}+b\frac{n}{2t^2}\\&+(1-\frac{2\alpha-2}{\alpha}
\lambda)d\frac{n}{t^2} +\alpha c \triangle R.
\end{align*}
\end{proof}

In the above lemma, let take $\alpha=2$, $a=-3$, $b=-1$, $c=-1$,
$\lambda=1$. Then we have
\begin{corollary}\label{coroP} Let $(M,g(t))$ be a solution to
the Ricci flow, $f$
be a positive solution of $$\ppt f=\triangle f+Rf,$$ let $v=-\ln f
-\frac{n}{2} \ln (4\pi t)$ and $$P=2 \triangle v-|\nabla
v|^2-3R+\frac{v}{t}-d\frac{n}{t}.$$ Then we have
\begin{align*}
\ppt P=&\triangle P-2\nabla P \cdot \nabla
v-2|v_{ij}-R_{ij}-\frac{1}{2t}g_{ij}|^2 -\frac{1}{t}P\\&
-2(\triangle R+2|Rc|^2+\frac{R}{t} +2 \nabla R \cdot \nabla v +2
R_{ij}v_iv_j),
\end{align*}
i.e.,
\begin{align}\label{panack}
\ppt (tP)=&\triangle (tP)-2\nabla (tP) \cdot \nabla
v-2t|v_{ij}-R_{ij}-\frac{1}{2t}g_{ij}|^2 \\& -2t(\triangle
R+2|Rc|^2+\frac{R}{t} +2 \nabla R \cdot \nabla v +2 R_{ij}v_iv_j).
\nonumber
\end{align}
\end{corollary}

Now we are ready to prove Theorem \ref{theorem1.2}.
\begin{proof} (Proof of Theorem \ref{theorem1.2})  It is easy to
see that the proof
 follows from Corollary \ref{coroP},
the trace Harnack inequality (\ref{tharnack}) for the Ricci flow
and the maximum principle.
\end{proof}

\section{\textbf{Entropy Formulas and Monotonicities}}

In this section, we will define two entropies which are similar to
Perelman's, and we will show that both of them are monotone under
the Ricci flow. Let $(M,g(t))$ be a solution to the Ricci flow on
a close manifold, and $f$ be a positive solution of (\ref{heat}).
Let $u=-\ln f$ and $H$ defined as in (\ref{defH}), we have

\begin{theorem} Assume that $(M, g(t))$ be a solution to the Ricci flow
with weakly positive curvature operator. Let  $$H=2 \triangle
u-|\nabla u|^2-3R-2\frac{n}{t}$$ and
$$F=\int_M t^2He^{-u} d\mu,$$ then $\forall t\in (0,T)$, we have
$F\leq 0$ and
$$\ddt F\leq 0.$$
\end{theorem}

\begin{proof} The fact that $F\leq 0$ follows directly from $H\leq
0$. We calculate its time derivative, using (\ref{harnack}) and
$\ppt d\mu=-Rd\mu$, we have
\begin{align*}\ddt F=&\int_M (2tHe^{-u}+t^2e^{-u}\ppt H+t^2H\ppt
e^{-u}-Rt^2He^{-u}) d\mu\\ =&\int_M [\triangle
(t^2He^{-u})-2t^2e^{-u}
|u_{ij}-R_{ij}-\frac{1}{t}g_{ij}|^2-2te^{-u}|\nabla
u|^2\\
&-2t^2e^{-u}(\ppt R +\frac{R}{t}+2 \nabla R \cdot \nabla u+2
R_{ij}u_iu_j)] d\mu \leq 0.
\end{align*}
\end{proof}

\begin{rema}
If we consider the system $$\left\{\begin{array}{l} \ppt
g_{ij}=-2(R_{ij}+\nabla_i \nabla_j u),\\ \ppt u=\triangle u-R,
\end{array}\right.$$ then the measure $dm=e^{-u}d\mu$ is fixed.
This system differs from the original system
$$\left\{\begin{array}{l} \ppt g_{ij}=-2R_{ij},\\ \ppt u=\triangle
u-|\nabla u|^2-R
\end{array}\right.$$ by a diffeomorphism.
\end{rema}

Now we consider $v=-\ln f-\frac{n}{2}\ln (4\pi t)$ and $P$ be
defined as in (\ref{defP}), we have
\begin{theorem} Assume that $(M, g(t))$ be a solution to the Ricci flow with
 weakly positive
curvature operator. Let $$P=2 \triangle v-|\nabla
v|^2-3R+\frac{v}{t}-d\frac{n}{t},$$ where $d$ is a constant. Let
$$W=\int_M tP(4\pi t)^{-n/2}e^{-v} d\mu,$$ then $\forall t\in
(0,T)$, we have
$$\ddt W\leq 0.$$
\end{theorem}

\begin{proof} Using (\ref{heat}), (\ref{panack}) and $\ppt d\mu=-Rd\mu$, we have
\begin{align*}\ddt W=&\int_M [P(4\pi t)^{-n/2}e^{-v}+t(4\pi t)^{-n/2}
e^{-v}\ppt P+tP\ppt ((4\pi t)^{-n/2}e^{-v})\\&-RtP(4\pi t)^{-n/2}e^{-v}] d\mu\\
=&\int_M [\triangle (tPe^{-v})-2te^{-v}
|v_{ij}-R_{ij}-\frac{1}{2t}g_{ij}|^2\\
&-2te^{-v}(\ppt R +\frac{R}{t}+2 \nabla R \cdot \nabla v+2
R_{ij}v_iv_j)] (4\pi t)^{-n/2} d\mu \leq 0.
\end{align*}
\end{proof}

\section{\textbf{A Gradient Estimate for the Heat Equation}}

In this section, we consider a special case of our general
evolution formula in Lemma \ref{lemma:H}. Let consider the
positive solution $f$ to the heat equation
\begin{equation}\label{heatf}
\ppt f=\triangle f,
\end{equation}
 since the equation is linear, without loss of
generality, we can assume that $0<f< 1$. Let $f=e^{-u}$, then $u$
satisfies
\begin{equation}\label{heatu}
\ppt u=\triangle u-|\nabla u|^2,
\end{equation} and $u> 0$.

In the proof of Lemma \ref{lemma:H}, let take $\alpha=0$,
$\beta=-1$, $a=c=0$, $b=1$ and $d=0$, then
$$H=|\nabla u|^2-\frac{u}{t},$$ and we have
\begin{align}\label{hgrad}
\ppt H=&\triangle H -2\nabla H \cdot \nabla u -\frac{1}{t}H-2
|\nabla \nabla u|^2,
\end{align}

\begin{theorem}\label{thmgrad}
Let $(M, g(t))$, $t\in [0, T)$,
 be a solution
to the Ricci flow on a closed manifold. Let $f(< 1)$ be a positive
solution to the heat equation (\ref{heatf}), $u=-\ln f$ and
$$H=|\nabla u|^2-\frac{u}{t}.$$ Then for
all time $t\in (0,T)$
$$H\leq 0.$$ Hence on $(0,T)$, $$|\nabla f|^2 \leq -\frac{f^2 \ln
f}{t}.$$
\end{theorem}

\begin{proof}
Notice that as $t$ small enough, $H<0$, now the proof follows from
(\ref{hgrad}) and the maximum principle.
\end{proof}

\begin{rema} In this case,
we do not need any curvature assumption.
\end{rema}

\begin{rema}
Q. Zhang told us that he has proved the same estimate as
our Theorem \ref{thmgrad} in \cite{zhang06}, using a similar idea of
the second author.
\end{rema}

\bibliographystyle{halpha}
\bibliography{bio}
\end{document}